\documentclass[a4paper,11pt]{article}

 \usepackage{amssymb}                   %includes \mathbb command
 \usepackage{amsmath}                   %includes \operatorname command
 \usepackage{amsthm}                    %includes proof environment
 \usepackage{latexsym}                  %includes \Box command

 \usepackage{amscd}
 \usepackage{amsmath}
 \usepackage[all]{xy}

 \newtheorem{thm}{Theorem}[section]

 \newtheorem{lem}[thm]{Lemma}
 
 \newtheorem{cla}[thm]{Claim}

 \theoremstyle{definition}
 
 \theoremstyle{remark}\newtheorem{remark}[thm]{Remark}
 
 \newtheorem{ex}[thm]{Example}

 \numberwithin{equation}{section}
 
\renewcommand{\labelenumi}{(\roman{enumi})}

 \newcommand{\RR}{\ensuremath{\mathbb{R}}} 
 \newcommand{\LL}{\ensuremath{\mathbb{L}}}
 \newcommand{\ltensor}{\overset{\LL}{\otimes}}
 \newcommand{\Aut}{\ensuremath{\operatorname{Aut}}}

 \newcommand{\Pic}{\ensuremath{\operatorname{Pic}}}
 \newcommand{\Spec}{\ensuremath{\operatorname{Spec}}}

 \newcommand{\Coker}{\ensuremath{\operatorname{Coker}}}
 \newcommand{\Image}{\ensuremath{\operatorname{Im}}}

 \newcommand{\td}{\ensuremath{\operatorname{td}}}
  \newcommand{\ch}{\ensuremath{\operatorname{ch}}}

 \newcommand{\rk}{\ensuremath{\operatorname{rk}}}
 \newcommand{\order}{\ensuremath{\operatorname{ord}}}
 \newcommand{\id}{\ensuremath{\operatorname{id}}}
 \newcommand{\wt}{\widetilde}
 \newcommand{\FM}{\operatorname{FM}}
 
 \newcommand{\mc}{\mathcal}
 \newcommand{\mb}{\mathbb}
 
 \newcommand{\rom}{\textrm}
 \newcommand{\wh}{\widehat}

 \newcommand{\fm}[1]{\ensuremath{\Phi^{#1}}}

\DeclareFontFamily{U}{UWCyr}{}
\DeclareFontShape{U}{UWCyr}{m}{n}{
   <5> <6> <7> <8> <9> gen * wncyr
   <10> <10.95> <12> <14.4> <17.28> <20.74> <24.88> wncyr10
   }{}
\DeclareMathAlphabet\cy{U}{UWCyr}{m}{n}

\newcommand{\Gal}{\operatorname{Gal}}
\newcommand{\perf}{\operatorname{perf}}
\newcommand{\PP}{\mathbb P}
\newcommand{\Z}{\mathbb Z}
\newcommand{\Q}{\mathbb Q}
\newcommand{\C}{\mathbb C}

\def\im{\mathop{\mathrm{im}}\nolimits}
\def\ker{\mathop{\mathrm{ker}}\nolimits}

\title{A counterexample of 
the birational Torelli problem via Fourier--Mukai transforms}  
\author{Hokuto Uehara}

\date{}

\begin{document}
\maketitle

\begin{abstract}
We study the Fourier--Mukai numbers of rational elliptic 
surfaces. As its application, we give an example of 
a pair of minimal  $3$-folds with 
Kodaira dimensions $1$, $h^1(\mc O)=h^2(\mc O)=0$ such that
they are mutually derived equivalent, deformation equivalent,
but not birationally equivalent. It also supplies a counterexample of 
the birational Torelli problem.  
\end{abstract}

%=======================================================================
%=======================================================================
 \section{Introduction}\label{section:introduction}
%=======================================================================
%=======================================================================

Let $X$ be a smooth projective variety over $\C$. The derived category $D^b(X)$ of $X$ is a triangulated category
whose objects are bounded complexes of coherent sheaves on $X$.
A \emph{Fourier--Mukai transform} relating smooth projective varieties 
$X$ and $Y$ is a $\C$-linear equivalence
of triangulated categories $\Phi: D^b(X)\to D^b(Y)$.
If there exists a Fourier--Mukai 
transform relating $X$ and $Y$, we call $X$ a
\emph{Fourier--Mukai partner} of $Y$, or simply say 
that $X$ and $Y$ are 
\emph{derived equivalent}.

It is an interesting problem to find a good characterization of  
Fourier-Mukai partners of given smooth projective varieties. 
For instance, it is known that 
two K$3$ surfaces are derived equivalent 
if and only if their Mukai lattices are Hodge isometric to each other 
(\cite{Or96}). We also have a moduli-theoretic characterization of  
Fourier-Mukai partners of certain minimal elliptic surfaces
due to Bridgeland and Maciocia
(see Theorem \ref{BMelliptic}).

 Furthermore it is also interesting to study the cardinality
 of the set of  
 isomorphism classes of Fourier--Mukai partners of $X$ (this set is often denoted by 
 $\FM (X)$ and its cardinality is called \emph{Fourier--Mukai number} of $X$).
 Although it is predicted that the Fourier--Mukai numbers of 
any smooth projective varieties are finite 
(\cite{BM01}, \cite{Ka02}, \cite{To06}),
it is known that there are no universal bounds of Fourier--Mukai numbers 
for the
families of K$3$ surfaces, abelian varieties, and rational  elliptic 
surfaces respectively
(\cite{Og02}, \cite{Or02}, \cite{HLOY03}, \cite{Ue04}).

%%%%%%%%%%%%%%%%%%%%%%%%%%%%%%%%%%%%%%%%%%%%%%%%%%%%%%%%%%%%%%%%%%%%%%%%%%%
%%%%%%%%%%%%%%%%%%%%%%%%%%%%%%%%%%%%%%%%%%%%%%%%%%%%%%%%%%%%%%%%%%%%%%%%%%%

\paragraph{Fourier--Mukai numbers of
rational  elliptic surfaces.}
In this article, we study the Fourier--Mukai numbers of
rational elliptic surfaces over $\C$.
Henthforth we consider only  
relatively minimal elliptic surfaces as elliptic surfaces.

Fix a rational elliptic surface
$\pi_0 \colon B\to \PP^1$ with a section and
a point $s\in\PP^1$, where
the fiber of $\pi_0$ over $s$ is of type $\rom{I}_n$ $(n\ge 0)$.
Choose an integer $m>1$ and 
apply a \emph{logarithmic transformation} to $B$ along $s$,
and then we obtain a rational  elliptic surface 
$\pi_1 \colon S\to \PP^1$ with 
a multiple fiber of type $_m\rom{I}_n$ over $s$
(cf.~Remark \ref{rem:rationalBS}). 
Furthermore it is known that every rational  elliptic 
surface $S$ has at most one multiple fiber,
and it is obtained  
by applying a logarithmic transformation to its Jacobian $J(S)$, 
which is again rational.

We have a bound of the Fourier--Mukai numbers of $S$
from below as follows.
In the statement, we denote by $\varphi(m)$ the Euler function.
 
%%%
%%%
%%%

\begin{thm}\label{thm:FMelliptic0}
Fix $B$ and $s\in \PP^1$ as above.
Take an integer $m>1$. 
Then we have a positive integer $n_0$, depending on $B$ and $s$, 
but not depending on $m$, satisfying
$$
\frac{\varphi (m)}{n_0} 
\leqq 
|\FM (S)|
$$ 
for any rational  elliptic surfaces 
$S$ obtained from $B$ and $m$
by a logarithmic transformation along the point $s$.
Consequently the Fourier--Mukai number of $S$ becomes larger
as we take a larger $m$.
\end{thm}

We remark that if $S$ has no multiple fiber  
or $S$ has a multiple fiber with multiplicity $2$, then 
we readily know that $|\FM (S)|=1$ (see Remark \ref{rem:FM1}). 
 
We can apply our method to compute the Fourier--Mukai 
numbers of certain rational  elliptic surfaces (see \S 
\ref{subsec:examples}).
Theorem \ref{thm:FMelliptic0} also produces
counterexamples to Kawamata's D-K conjecture (\cite{Ka02}) as in \cite{Ue04}.

%%%%%%%%%%%%%%%%%%%%%%%%%%%%%%%%%%%%%%%%%%%%%%%%%%%%%%%%%%%%%%%%%%%%%%%%%%%
%%%%%%%%%%%%%%%%%%%%%%%%%%%%%%%%%%%%%%%%%%%%%%%%%%%%%%%%%%%%%%%%%%%%%%%%%%%

\paragraph{Minimal  $3$-folds and the birational Torelli problem.}
The second aim of this article is to study certain minimal  $3$-folds
in the contexts of derived categories and the Torelli problem. 
Let us consider the fiber products $X$ of 
two rational  elliptic surfaces over $\PP ^1$
with some properties.
Using Theorem \ref{thm:FMelliptic0}, 
we study Fourier--Mukai numbers of such $3$-folds $X$.
The precise statement is the following.

%%%
%%%
%%%

\begin{thm}\label{thm:main20}
Let $N$ be a given positive integer.
Then there are smooth minimal  $3$-folds $X_i$ $(i=1,\ldots ,N)$ 
satisfying the following properties: 
\begin{enumerate}
\item
For all $i$, $\kappa (X_i)=1$ and $X_i$'s have the following 
Hodge diamond:
$$
\begin{array}{ccccccccc}
 &  &  & 1&  &  &  &  & \\
 &  & 0&  & 0&  &  &  & \\
 & 0&  &19&  & 0&  &  & \\
1&  &19&  &19&  & 1&  & \\
 & 0&  &19&  & 0&  &  & \\
 &  & 0&  & 0&  &  &  & \\
 &  &  & 1&  &  &  &  & \\ 
\end{array}
$$
\item
$X_i$ and $X_j$ are not birationally 
equivalent for $i\ne j$.
\item
All $X_i$'s are mutually deformation equivalent 
and derived equivalent.
\item
For all $i,j$, we have Hodge isometries
$$
(H^3(X_i,\Z)_{\text{free}},Q_{X_i})\cong (H^3(X_j,\Z)_{\text{free}},Q_{X_j}),
$$
where the polarizations are given by the intersection forms.
\end{enumerate}

In particular, they supply a counterexample to the birational 
Torelli problem.
\end{thm}

%%%
%%%
%%%

Bridgeland \cite{Br02} shows that
two smooth projective $3$-folds connected by a sequence of flops
are derived equivalent.
Consequently birationally equivalent 
smooth minimal  $3$-folds are derived 
equivalent. Motivated by his result,
 Borisov and C\u{a}ld\u{a}raru in \cite{BC09} show
 that there is a pair of 
Calabi--Yau $3$-folds such that they are derived equivalent but 
not birationally equivalent. Our theorem assures that 
a similar phenomenon happens for the case of Kodaira dimension $1$.
Furthermore in \cite{Ca07} (see also \cite[Conjecture 0.2]{Sz04}), 
C\u{a}ld\u{a}raru attempts to construct counterexamples to the birational 
Torelli problem for  Calabi--Yau $3$-folds. 
Another counterexample is discovered by Namikawa in \cite{Na02} 
for irreducible symplectic manifolds.
Our result says that the birational 
Torelli problem fails for the above minimal $3$-folds.

The Iitaka fibrations of the above $3$-folds $X_i$'s   
have multiple fibers. So the failure of the birational 
Torelli problem may not be very surprising because 
a similar phenomenon occurs for the $2$-dimensional case (\cite{Ch80}, 
see also Remark \ref{rem:Chakiris}).

%%%%%%%%%%%%%%%%%%%%%%%%%%%%%%%%%%%%%%%%%%%%%%%%%%%%%%%%%%%%%%%%%%%%%%%%%%%
%%%%%%%%%%%%%%%%%%%%%%%%%%%%%%%%%%%%%%%%%%%%%%%%%%%%%%%%%%%%%%%%%%%%%%%%%%%

\paragraph{Construction of this article.}
In \S \ref{section:elliptic surfaces}, first we recall some general facts of 
Fourier--Mukai partners of elliptic surfaces and 
the Ogg--Shafarevich theory. After that we give the proof of 
Theorem \ref{thm:FMelliptic0}.
In \S \ref{section:3folds}, we first show some easy lemmas on 
Fourier--Mukai transforms between varieties of fiber products.
Secondly by taking the fiber products of 
rational  elliptic surfaces with certain properties,
we construct minimal  $3$-folds in Theorem \ref{thm:main20}.

\paragraph{Acknowledgments.}
The referee pointed out some mistakes and simplifications 
of the original proof.
I would like to thank the referee for invaluable suggestions.
I am supported by the Grants-in-Aid 
for Scientific Research (No.20740022).

%%%%%%%%%%%%%%%%%%%%%%%%%%%%%%%%%%%%%%%%%%%%%%%%%%%%%%%%%%%%%%%%%%%%%%%%%%%
%%%%%%%%%%%%%%%%%%%%%%%%%%%%%%%%%%%%%%%%%%%%%%%%%%%%%%%%%%%%%%%%%%%%%%%%%%%

\paragraph{Notations and conventions.}
All varieties are defined over $\C$,
and \emph{elliptic surface}
always means \emph{relatively minimal elliptic surface}.

A point on a variety means a closed point unless specified
 otherwise.
  
Let $\pi\colon X\to Y$  
be a surjective projective morphism
between smooth projective varieties $X$ and $Y$. For a point $t\in Y$, 
we denote the scheme-theoretic fiber of $t$ by $X_t$ and
the discriminant locus of $\pi$ by $\Delta(\pi)$.

We denote the diagonal with reduced structure 
in $X\times X$ by $\Delta _X$.

Let $\pi\colon B\to C$ be an elliptic 
surface with the $0$-section
and $s$ a point on $C$. 
We denote by $\Aut _0 B$  the group  
consisting of the automorphisms $\gamma$ of $B$ 
which fix the $0$-section as a curve,
and make the following diagram commutative for 
some automorphisms $\delta$ of $C$:
\[ \xymatrix{ B \ar[d]_{\pi} \ar[r] ^{\gamma} & B \ar[d]^{\pi} \\  C \ar[r]^{\delta} & C}\]
Furthermore $\Aut _0(B,s)$ (resp.~$\Aut _0(B/C)$ ) 
is the group consisting of $\gamma\in \Aut _0B$
which induces $\delta\in \Aut C$
 fixing the point $s\in C$
(resp.~all points in $C$).

 For a set $I$, we denote by $|I|$ the cardinality of $I$.

%========================================================================
%========================================================================
\section{Rational  elliptic surfaces}\label{section:elliptic surfaces}
%========================================================================
%========================================================================

\subsection{Fourier--Mukai partners of elliptic surfaces}
\label{sub:FMellipitic}

We need some standard notation and results before going further.
Let $\pi:S\to C$ be an elliptic surface.
For an object $E$ of $D^b(S)$, we define the fiber degree of $E$ 
\[d(E)=c_1(E)\cdot f, \]
where $f$ is a general fiber of $\pi$. Let us denote by $\lambda_{S/C}$  
the highest common factor of the fiber degrees of objects of $D^b(S)$. 
Equivalently,
$\lambda_{S/ C}$ is the smallest number $d$ such that there is a 
holomorphic $d$-section of $\pi$. 
For integers $a>0$ and $i$ with $i$ coprime to $a\lambda_{S/ C}$, 
by \cite{Br98} there exists a smooth,
2-dimensional component $J_S (a,i)$ of the moduli space of pure dimension 
one stable sheaves on $S$,
the general point of which represents a rank $a$, degree $i$ stable 
vector bundle supported on a smooth fiber of $\pi$. 
There is a natural morphism $J_S (a,i)\to C$, taking a point representing
 a sheaf supported on the
fiber $\pi ^{-1}(x)$ of $S$ to the point $x$. This morphism is a minimal 
elliptic fibration (\cite{Br98}).
Put $J^i(S):=J_S(1,i)$. 
Obviously, $J^0(S)\cong J(S)$, the Jacobian surface associated to $S$, 
and $J^1(S)\cong S$. 
Moreover there is a natural isomorphism $J^i(S)\cong J^{i+\lambda_{S/C}}(S)$.
Hence we may regard $i$ as an element of $\Z/\lambda_{S/C}\Z$, 
instead of $\Z$,
when we consider the isomorphism classes of $J^i(S)$. 
We have a nice characterization of 
Fourier--Mukai partners of elliptic surfaces with non-zero Kodaira dimensions:

%%%%%%%%%
%%%%%%%%%
%%%%%%%%%

\begin{thm}[Proposition 4.4 in \cite{BM01}]\label{BMelliptic}
Let $\pi :S\to C$ be an elliptic surface and $T$ a smooth projective variety.
Assume that the Kodaira dimension $\kappa (S)$ is non-zero.
Then the following are equivalent.
\renewcommand{\labelenumi}{(\roman{enumi})}
\begin{enumerate}
\item 
$T$ is a Fourier--Mukai partner of $S$. 
\item 
$T$ is isomorphic to $J^b(S)$ for some integer $b$ with $(b,\lambda _{S/C})=1$. 
\end{enumerate}
\end{thm}

%========================================================================
\subsection{Weil--Ch\^atelet group}\label{subsection:WC}

Fix an elliptic surface with a section 
$\pi_0\colon B\to C$. Let $\eta=\Spec k$ be the 
generic point of $C$,
where $k=k(C)$ is the function field of $C$, and let $\overline{k}$ be an 
algebraic closure of $k$.
Put $\overline{\eta}=\Spec \overline{k}$. We define the 
\emph{Weil--Ch\^atelet group} $WC(B)$
by the Galois cohomology $H^1(G, B_{\eta}(\overline{k}))$. 
Here $G=\Gal(\overline{k}/k)$ and $B_{\eta}(\overline{k})$ 
is the group of points of the elliptic curve $B_{\eta}$ defined over 
$\overline{k}$. 

Suppose that we are given a pair $(S,\varphi)$, where $S$ is 
an elliptic surface $S\to C$ and 
$\varphi$ is an isomorphism $J(S)\to B$ over $C$, fixing their $0$-sections. 
Then we have a morphism 
$$
B_{\eta}\times_k S_{\eta}\to J(S_{\eta})\times_k  S_{\eta}\to S_{\eta}.
$$
Here the first morphism is induced by $\varphi^{-1}\times id_S$ and 
the second is given by translation.
We obtain a principal homogeneous space $S_{\eta}$ over $B_{\eta}$.
Fix a point $p\in S_\eta(\overline{k})$ and 
take an element $g\in G$. Then the element $g(p)-p\in J(S_\eta)(\overline{k})$ 
can be regarded as an element
of $B_\eta(\overline{k})$ via $\varphi$. 
The map
$$
G\to B_\eta(\overline{k}) \qquad g\mapsto g(p)-p 
$$
is a $1$-cocycle and changing 
a point $p$ replaces it by a $1$-coboundary. 
Therefore this map defines a 
class in $WC(B)$.
Since this correspondence is invertible (cf.~\cite{Se02}),
% and 
%the group $H^1(G, B_{\eta}(\overline{k}))$ classifies isomorphism classes 
%of principal homogeneous spaces over $B_{\eta}$ (cf. \cite[page 184]{Fri98}
we know that $WC(B)$ consists of all isomorphism
 classes of pairs $(S,\varphi)$.
Here two pairs $(S,\varphi)$ and $(S',\varphi ')$ are \emph{isomorphic} 
if there is an isomorphism 
$\alpha :S\to S'$ over $C$, such that $\varphi '\circ \alpha _* =\varphi$,
where $\alpha _*:J(S)\to J(S')$ is the isomorphism induced by $\alpha$ 
(fixing $0$-sections). 
\noindent
\[ \xymatrix{ J(S) \ar[d]_{\varphi} \ar[r] ^{\alpha _*} & J(S') \ar[d]^{\varphi '} \\  B \ar@{=}[r] & B}\]
\noindent

For any $\xi:=(S,\varphi_1)\in WC(B)$, $g\in G$ and $i\in \Z$, 
we obtain an element $i(g(p)-p)\in B_\eta(\overline{k})$, which can be regarded 
as an element of $J(J^i(S_\eta))(\overline{k})$. Therefore
there is an isomorphism $\varphi_i: J(J^i(S))\to B$ such that 
$(J^i(S),\varphi_i)$ defines the class $i\xi\in WC(B)$. %(\cite[page 38]{Fri95}).

%For simplicity, assume that a minimal ellipitic surface $S$ is rational and define $B=J(S)$.
The group $\Aut _0(B)$ acts on $WC(B)$ as follows:
Let $\delta\in \Aut C$ be the automorphism on $C$ 
induced by $\gamma\in\Aut _0(B)$. Then for 
$$
\xi=(\pi_1\colon S\to C,\varphi_1)\in WC(B),
$$
we define 
\begin{align}\label{eqn:action}
\gamma \xi:=(\delta \circ \pi_1\colon S\to C,\gamma \circ \varphi_1).
\end{align}

Suppose that $\kappa (S)\ne 0$. Then
Theorem \ref{BMelliptic} implies that the map 
\begin{equation}\label{eq:WCFM} 
\Phi \colon
\bigl\{ i\xi \in WC(B) \bigm| i\in (\Z/\lambda_{S/C}\Z)^* \bigr\}\to \FM (S)
\quad   i\xi=(J^i(S),\varphi_i) \mapsto J^i(S)
\end{equation}
is surjective.
For the involution $\gamma\in \Aut_0(B/C)$, 
we shall see in \S \ref{subsec:examples} that 
%$$(J^i(S),\gamma\circ \varphi _i)=
$$\gamma\xi=-\xi.$$
In particular, the map $\Phi$ is not injective whenever
$\lambda_{S/C}>2$. Therefore it is important to investigate the 
preimages of $\Phi$ when 
we study the Fourier--Mukai number of $S$.
 
%Take an element $(\pi_1\colon S\to C,\varphi)\in WC(B)$ and  a closed point $t\in C$ such that $S_t$ is not a multiple fiber. Then $B_t$ and $S_t$ are isomorphic each other (cf. \cite[page 46]{FM94}).  Hence we have 
%
%\begin{equation}\label{eqn:Delta}
%\Delta(\pi_0)\cup \{s_1,\ldots,s_n\}=\Delta(\pi_1),
%\end{equation}
%  
%where $s_i$'s are the points over which $\pi_1$ has multiple fibers. 

%========================================================================
\subsection{Local invariant}\label{subsec:local}

Consider the completion $\wh{\mc O_{C,t}}$ of the local ring $\mc O_{C,t}$ for $t\in C$.
 We denote by $K_t$ its field of fraction and put $\wt{B_t}:=B\times_C \Spec K_t$
and $\wt{S_t}:=S\times_C \Spec K_t$ for $\xi=(S,\varphi)\in WC(B)$.
If $S_t$ is not a multiple fiber, $S_t$ has a reduced irreducible component. Thus
Hensel's lemma implies that there is a section of 
$\wt{S_t} \to \Spec K_t$ and then $\wt{S_t}$ 
is a principal homogeneous
space over $\wt{B_t}$. Put $WC(B_t):=H^1(G_t,\wt{B_t}(\overline{K_t}))$, where
$G_t$ is the Galois group of the local field $K_t$ and $\overline{K_t}$ is an 
algebraic closure of ${K_t}$.
 Denote 
by $\xi_t$ the class in $WC (B_t)$ induced by $\xi$.
Then
there is a group homomorphism 
\begin{equation}\label{eq:BtoBt}
WC(B) \to \bigoplus _{t\in C} WC(B_t) \qquad \xi\mapsto (\xi _t)_{t\in C},
\end{equation}
which is compatible with the natural group action of $\Aut_0 B$ (cf.~(\ref{eqn:action})).
The element $\xi_t$ is called \emph{local invariant} at $t$ and  
the kernel of the map (\ref{eq:BtoBt}) is called  \emph{Tate--Shafarevich group}
and denote it by $\cy{Sh}(B)$. Namely,
$\cy{Sh}(B)$ is the subgroup of $WC(B)$ 
which consists of all isomorphism classes of pairs 
$(\pi_1\colon S\to C,\varphi_1)$ such that 
$\pi_1$ has no multiple fibers. 

It is known that for $\xi=(\pi_1\colon S\to C,\varphi_1)\in WC(B)$, 
$\pi_1$ has a multiple fiber of multiplicity $m$ over 
$s\in C$ if and only if $\xi_s$ in (\ref{eq:BtoBt}) has order $m$.
Moreover the map in (\ref{eq:BtoBt}) is surjective
if $B$ is not the product $C\times E$, where $E$ is an elliptic curve
(%\cite[page 185]{Fri98}, % \cite[page 38]{Fri95},
 cf. \cite[Theorem 5.4.1]{CD89}). 

We have the following isomorphisms (\cite[page 124]{Do81});
%and fix them below;
%
\begin{align}
WC(B_{s})& \cong H^1(B_s,\Q/\Z) \label{eq:localWC1} \\
&\cong
\begin{cases}
\Q/\Z\oplus \Q/\Z \quad \text{ if $B_s$ is a smooth elliptic curve, }\\
\Q/\Z \quad \text{ if $B_s$ is a singular fiber of type $\rom{I}_n$ $(n>0)$, }\\0 \quad \text{ otherwise. }\\
\end{cases}\label{eq:localWC2}
\end{align}
The group $\Aut _0(B,s)$ naturally acts on both sides of
(\ref{eq:localWC1}) and the isomorphism in (\ref{eq:localWC1})
is equivariant 
under these actions.

%=======================================================================
 \subsection{Proof of Theorem \ref{thm:FMelliptic0}}
 \label{section:rational}
%=======================================================================
Throughout in this subsection,
we denote by $\pi _0\colon S\to \PP^1$  a rational  elliptic surface
with a fiber of type $_m\rom{I}_n$ over a point $s\in\PP ^1$ 
for $m>1$, and denote its Jacobian by $\pi_0\colon B\to \PP ^1$.

\begin{remark}\label{rem:rationalBS}
For a rational  elliptic surface $S$, $B=J(S)$ is again rational.
Conversely, starting from a rational  elliptic surface $B$, 
we obtain a rational surface $S$ by a logarithmic transformation 
along a single point $s$. (To show these statements, see, for instance,   
\cite[Proposition 1.3.23, Theorem 1.6.7]{FM94}.)
In particular, $J^i(S)$ is also rational, since 
$J^i(S)$ is obtained from $B$ by a logarithmic transformation. 
\end{remark}
%then we can regard the pair $(S,\id_B)\in WC(B)$ as the element of 
%the form $(p/m,q/m)\in \Q/\Z$ for $p,q\in \Z_{>0}$, 
%Moreover $S$ is obtained from $B$ by a 
%logarithmic transformation along a fiber of type $\rom{I}_n$ $(n\ge 0)$.
Any automorphisms 
of $B$ induce automorphisms of $\PP^1$, 
since the rational surface 
$B$ has a unique elliptic fibration. 
Hence we have natural homomorphisms 
$$
\Aut S\to \Aut _0(B,s)\to \Aut \PP^1 .
$$

\begin{lem}\label{lem:naka}
The group 
$$
\Image (\Aut _0(B,s)\to \Aut \PP ^1 )
\cong  
\Aut _0(B,s)/\Aut _0(B/\PP ^1)$$
is finite. 
\end{lem}

\begin{proof}
Recall that the fiber of $\pi_0$ over the point $s$ is of 
type $\rom{I}_n$ for $n\ge 0$.
By a quick view of Persson's list \cite{Pe90}, 
we know that $|\Delta (\pi_0)\backslash \{s\}|\ge 2$.   
Since the group 
$
\Image (\Aut _0(B,s)\to \Aut \PP ^1 )
$ 
preserves the point $s$ and the set $\Delta (\pi_0)\backslash \{s\}$,
it is finite. 
\end{proof}

We put 
%\begin{align*}
$$
N_1:=\Image (\Aut _0(B,s)\to \Aut \PP ^1 ), \quad 
N_2:=\Coker (\Aut S\to N_1)
$$
%\end{align*}
and call their cardinalities
$
n_1,n_2
$
respectively. We define $\xi:=(S,\id _B)\in WC(B)$.

Now we are in position to show Theorem \ref{thm:FMelliptic0}.
\newline

%%%
%%%
%%%  Proof of Theorem 1.1.

\noindent
\emph{Proof of Theorem \ref{thm:FMelliptic0}.}
First of all, we note that $\lambda _{S/\PP^1}=m$, since
every $(-1)$-curve on $S$ is a $m$-section of $\pi_1$.
By the definition of the map $\Phi$ in (\ref{eq:WCFM}), we have
the natural one to one correspondence between the set 
$\Phi^{-1}(J^i(S))$
and the set 
$$
I(i):=\bigl\{k\in (\Z/m\Z)^* \bigm| J^i(S)\cong J^k(S)\bigr\}
$$ 
for any $i\in (\Z/m\Z)^*$. 

Because there is an isomorphism $J^j(J^i(S))\cong J^{ij}(S)$
 for $i,j\in \Z/m\Z$, 
$S$ is isomorphic to a surface $T$ if and only if
$J^i(S)\cong J^i(T)$ for some $i\in (\Z/m\Z)^*$. 
Then we have the equality $|\Phi^{-1}(J^1(S))|=|\Phi^{-1}(J^i(S))|$
 for any $i\in (\Z/m\Z)^*$. 
Therefore we know that 
$$
|\FM (S)|=\frac{|(\Z/m\Z)^*|}{|\Phi ^{-1}(J^1(S))|}=\frac{\varphi(m)}{|I(1)|}.
$$

Henceforth we identify $S$ and $J^1(S)$ by the natural isomorphism between 
them. 
For each $k\in I(1)$, we fix an isomorphism $\alpha_k:S \to J^k(S)$. 
Because both of $S$ and $J^k(S)$ are rational,
$\alpha_k$ induces an automorphism $\delta_k$ of 
$\PP^1$ such that the following diagram is commutative:
\noindent
\[ \xymatrix{ S \ar[d]_{\pi_1} \ar[r]^{\alpha_k}& J^k (S) \ar[d]^{\pi_k} \\  \PP ^1 \ar[r] ^{\delta_k} & \PP ^1}\]
Consider the following automorphism $\gamma_k\in \Aut_0(B,s)$; 
$$
\gamma_k:=\varphi_k\circ\alpha _{k *}\colon 
B= J(S)\to  J(J^k (S))\to B,
$$
where we define $\varphi_k$ as $k\xi=(J^k(S),\varphi_k)$ holds 
(see \S \ref{subsection:WC}).
Hence for any $k\in I(1)$,
we can find an element $\gamma_k$ of $\Aut _0(B,s)$ 
such that $\gamma_k\xi=k\xi$ (see (\ref{eqn:action})). In particular
we have
\begin{equation}\label{eq:I(1)}
I(1)=\bigl\{ 
k\in  (\Z/m\Z)^*  \bigm|
k\xi \in \Aut _0(B,s)\xi
\bigr\}
\end{equation} 
and hence
$$
I(1)\le |\Aut _0(B,s)\xi|.
$$ 
Next we show the following;
\begin{equation}\label{eqn:len2}
|\Aut _0(B,s)\xi|\le n_2|\Aut _0(B/\PP ^1)|.
\end{equation}
Let us define two equivalent conditions on the group $\Aut_0(B,s)$:
For $\gamma_1,\gamma_2\in \Aut_0(B,s)$, define
\begin{align*}
\gamma_1 \sim _1 \gamma_2 &\iff \gamma_1\xi=\gamma_2\xi,\\
\gamma_1 \sim _2 \gamma_2 &\iff \gamma\gamma_1\xi=\gamma_2\xi
\mbox{ for some } \gamma \in \Aut_0(B/\PP ^1). 
\end{align*}
By the definition, we have 
\begin{align*}
\gamma_1\sim_2\gamma_2
&\iff
\gamma \gamma_2^{-1}\gamma_1= \alpha_*
\mbox{ for } \alpha \in \Aut S \mbox{ and }
\gamma \in \Aut_0(B/\PP ^1)
\end{align*}
But these conditions are also equivalent to the fact that
the automorphism of 
$\PP ^1$ induced by 
$\gamma_2^{-1}\gamma_1$ belongs to
the group 
$$
\Image (\Aut S\to \Aut \PP ^1).
$$
Therefore there is a one-to-one correspondence between
the set $\Aut _0(B,s)/\sim_2$ and the set $N_2$.
By the definitions of $\sim_1$ and $\sim_2$,
 at most $|\Aut_0 (B/\PP^1)|$ elements in  $\Aut _0(B,s)/\sim_1$
correspond to a single element in  $\Aut _0(B,s)/\sim_2$. 
Hence (\ref{eqn:len2}) follows.  

Put $n_0:=n_1|\Aut _0(B/\PP ^1)|$, 
then we obtain 
$$
\frac{\varphi(m)}{n_0}\le
\frac{\varphi(m)}{n_2|\Aut _0(B/\PP ^1)|}\le
\frac{\varphi(m)}{|I(1)|}=|\FM(S)|.
$$ 
Obviously the integer $n_0$ is independent on the choice of $m$.

\qed

\begin{remark}\label{rem:FM1}
For a minimal  rational elliptic surface $S$ 
without multiple fibers or with a multiple fiber of multiplicity $2$,
we can readily see from Theorem \ref{BMelliptic} 
that the Fourier--Mukai partner of $S$ is only itself, i.e.
$$|\FM(S)|=1.$$
\end{remark}

%%%%%%%%%%%%%%%%%%%%%%%%%%%%%%%%%%%%%%%%%%%%%%%%%%%%%%%%%%%%%%%%%%%%%%%%%%%
%%%%%%%%%%%%%%%%%%%%%%%%%%%%%%%%%%%%%%%%%%%%%%%%%%%%%%%%%%%%%%%%%%%%%%%%%%%

\subsection{Examples}\label{subsec:examples}
Our method in the proof of Theorem \ref{thm:FMelliptic0} sometimes 
gives the Fourier--Mukai numbers of some rational  
elliptic surfaces.
Here we use the same notation as in the previous subsection.

Since $B$ is also rational, 
$\cy{Sh}(B)$ is trivial (cf. \cite[Example 1.5.12]{FM94}, \cite[Corollary 5.4.9]{CD89}).  
In particular, the homomorphism (\ref{eq:BtoBt}) becomes an isomorphism
\begin{equation}\label{eq:rational}
WC(B)\cong \bigoplus_{t \in \PP ^1} WC(B_{t}) \qquad \xi\mapsto (\xi_t)_{t\in \PP^1}
\end{equation}
%
%Hence if $\pi_1$ has no multiple fiber, equivalently $S\cong J(S)$, 
%the element $\xi:=(S,\id_B)$ is trivial in the group $WC(B)$.  
%$\bigoplus_{t \in \PP ^1} WC(B_{t}).$
Because $\pi_1$ has a multiple fiber of type $_m\rom{I}_n$ 
over the point $s\in \PP ^1$, 
 the element $\xi=(S,\id _B)\in WC(B)$ is completely determined by an 
element $\xi_s\in WC(B_s)$ of the form
$$
\xi_s=
\begin{cases}
(p/m,q/m) \in \Q/\Z\oplus\Q/\Z \quad \text{ in the case $n=0$, }\\
p/m \in \Q/\Z \quad \text{ in the case $n>0$. }\\
\end{cases}
$$
Here we identify $WC(B_s)$ with $\Q/\Z\oplus\Q/\Z$ or $\Q/\Z$ by
the isomorphisms (\ref{eq:localWC1}) and (\ref{eq:localWC2}),
and 
recall that the order of $\xi_s$ is $m$.
%In particular, $(S,\id_B)$ corresponds to 
%for $p,q\in \Z_{>0}$, and

We put
\begin{align*}
I':&=\bigl\{ 
k\in  (\Z/m\Z)^*  \bigm|
k\xi_s \in \Aut _0(B/\PP^1)\xi_s 
\bigr\}\\
&(=\bigl\{ 
k\in  (\Z/m\Z)^*  \bigm|
k\xi \in \Aut _0(B/\PP^1)\xi 
\bigr\}).
\end{align*}

\paragraph{Case: $B_s$ is a smooth elliptic curve.}
Observing the isomorphisms (\ref{eq:localWC1}) and (\ref{eq:localWC2})
carefully,
we know that a generator $\gamma$ of $\Aut_0(B/C)$ acts on   
$
WC(B_{s})\cong \Q/\Z\oplus \Q/\Z
$
as 
\begin{equation}\label{eq:gamma}
\gamma (a,b)=\begin{cases}
 (-a,-b)   & \text{ if } \order\gamma=2\\
  (-b,a)   &  \text{ if }\order\gamma=4\\
(-b,a+b) &  \text{ if }\order\gamma=6,
\end{cases}
\end{equation}
where $a,b\in \Q/\Z$.

\begin{ex}\label{ex:46}
\begin{enumerate}
\item
Consider the case $|\Aut_0(B/\PP^1))|=4$ and take
$\xi_s=(1/5,3/5)$. Then $|I'|=4$.
\item
Consider the case $|\Aut_0(B/\PP^1))|=6$ and take
 $\xi_s=(1/7,4/7)$. Then $|I'|=6$.
\end{enumerate}
\end{ex}

\paragraph{Case: $B_s$ is a singular fiber of type $\rom{I}_n$ $(n>0)$.}
A generator $\gamma$ of $\Aut_0(B/C)\cong \Z/2\Z$ acts on   
$
WC(B_{s}) \cong \Q/\Z
$
as 
\begin{equation}
\gamma a=-a
\end{equation}
for $a\in \Q/\Z$.
Therefore we obtain
$I'=\{1,-1\}(\subset (\Z/m\Z)^*)$. 

\paragraph{Assumption $n_1=1$.}
Assume that $n_1=1$, which also implies $n_2=1$.
Then by the proof of Theorem \ref{thm:FMelliptic0}, 
we know that the set $\Aut _0(B,s)/\sim_2$ has a single element.
Hence
for any $\gamma_1\in \Aut_0(B,s)$, there is an automorphism
$\gamma\in \Aut_0(B/\PP ^1)$ such that
$\gamma\xi= \gamma_1\xi$.  
This implies 
\[
I'=\bigl\{ 
k\in  (\Z/m\Z)^*  \bigm|
k\xi \in \Aut _0(B,s)\xi 
\bigr\},
\]
and hence $I'=I(1)$ by (\ref{eq:I(1)}).
Therefore we obtain the formula;
$$
|\FM(S)|=\frac{\varphi (m)}{|I'|}.
$$

As mentioned above (cf.~(\ref{eq:rational})), every rational  
elliptic surface $\pi_1\colon S\to \PP^1$ is determined by its local 
invariant $\xi_s$ in $WC(B_s)$.  
On the other hand, $|I'|$ can be easily computed 
if we are given the local invariant $\xi_s$. 
In particular, if $n_1=1$ for some fixed $B$ and $s\in \PP^1$, 
the Fourier-Mukai number
$|\FM(S)|$ with $B=J(S)$ is computable. 

%%%
%%%
%%% example

\begin{ex}
We can actually find the following rational  
elliptic surfaces $\pi _0\colon B\to \PP^1$ 
in the Persson's list \cite{Pe90}.
As we see below, they satisfy $n_1=1$ for some $s\in\PP^1$. 
We have a lot of choices of $B$ and $s\in \PP^1$ 
satisfying $n_1=1$ by \cite{Pe90}.
\begin{enumerate}
\item
Suppose that $\pi_0$ has two singular fibers of types $\rom{II}^*$ and $\rom{II}$.
Choose a point $s$ such that $B_s$ is smooth. 
Then any elements in the group $N_1$ should fix all three points in the set
 $\Delta(\pi_0)\cup \{s\}$.
Therefore it must be the identity, i.e. $n_1=1$. In this case, the $J$-map has the constant
 value $0$ and $|\Aut _0(B/\PP^1)|=6$.
Define $\xi_s: =(1/7,4/7)\in WC(B_s)$ and denote by $S$
the rational  elliptic surface
 corresponding to $\xi_s$. Then Example \ref{ex:46}(ii) implies 
$$
|\FM (S)|=\frac{\varphi (7)}{6}=1.
$$ 
\item
Suppose that $\pi_0$ has a singular fiber of type $\rom{II}^*$ and two singular fibers of
type $\rom{I}_1$. Take a point $s$ such that $B_s$ is of type $\rom{I}_1$.
Then by the same reason as above, we conclude $n_1=1$.
In this case, $|\Aut _0(B/\PP^1)|=2$. Apply a logarithmic transformation
along the point $s$ and then we obtain 
a rational  elliptic surface $S_m$ whose Jacobian surface is $B$,
and $S_m$ has a multiple fiber of type $_m\rom{I}_1$ over the point $s$ 
for some $m>0$.
Assume $m>2$ to assure $1\ne -1$ in $(\Z/m\Z)^*$. Then we know  
$$
|\FM (S_m)|=\frac{\varphi (m)}{2}.
$$ 

\end{enumerate}
\end{ex}

%%%%%%%%%%%%%%%%%%%%%%%%%%%%%%%%%%%%%%%%%%%%%%%%%%%%%%%%%%%%%%%%%%%%%%%%%%%%%%%
%%%%%%%%%%%%%%%%%%%%%%%%%%%%%%%%%%%%%%%%%%%%%%%%%%%%%%%%%%%%%%%%%%%%%%%%%%%%%%%
\section{Smooth minimal  $3$-folds with $\kappa (X)=1$}\label{section:3folds}
%%%%%%%%%%%%%%%%%%%%%%%%%%%%%%%%%%%%%%%%%%%%%%%%%%%%%%%%%%%%%%%%%%%%%%%%%%%%%%%
%%%%%%%%%%%%%%%%%%%%%%%%%%%%%%%%%%%%%%%%%%%%%%%%%%%%%%%%%%%%%%%%%%%%%%%%%%%%%%%

\subsection{Fourier--Mukai transforms between varieties of fiber products}
The results in this subsection  must be well-known to specialists.
Let $X$ and $Y$ be smooth projective varieties and
$\mc U$ an object in $D^b(X\times Y)$.
Then the object $\mc U$ determines the integral functor
$$
\fm{\mc U}:=\RR\pi _{Y*}(\pi^*_X(-)\ltensor \mc U)\colon D^b(X)\to D^b(Y),
$$
where $\pi_X\colon X\times Y\to X$ and $\pi_Y\colon X\times Y \to Y$
are projections.
We call $\mc U$ the kernel of the functor.
If $\fm{\mc U}$ gives an equivalence, 
we call it a \emph{Fourier--Mukai transform}. 

Next suppose that we are given a closed integral subvariety 
$\iota \colon Z\hookrightarrow  X\times Y$
and a perfect object $\mc U_Z\in D_{\perf}(Z)$. 
Moreover assume that the restrictions 
$\pi_X |_Z$ and $\pi_Y|_Z$ are flat.
Then the projection formula yields that the functor 
$$
\RR(\pi _{Y}|_Z)_*((\pi_X|_Z)^*(-)\ltensor \mc U_Z)\colon D^b(X)\to D^b(Y)
$$
is isomorphic to the functor
$\fm{\iota_*\mc U_Z}$.
By the abuse of notation, we also denote it by 
$\fm{\mc U_Z}$.
 
Let $X_0$ and $Y_0$ be smooth closed subvarieties of $X$ and $Y$ 
respectively and subvariety $Z_0(\subset Z)$ is 
a scheme-theoretic pull-back of $X_0$ and $Y_0$ by $\pi_X|_Z$ and $\pi_Y|_Z$.
We consider the following diagram:
\[ 
\xymatrix{ 
X  &    Z \ar[l]_{\pi_X|_{Z}} \ar[r]^{\pi_Y|_{Z}}  & Y    \\
X_0 
\ar@{^{(}->}[u]^{\iota_X} &  Z_0 \ar[l]^{\pi_X|_{Z_0}} \ar[r]_{\pi_Y|_{Z_0}}
\ar@{^{(}->}[u]^{\iota_Z} & Y_0 
\ar@{^{(}->}[u]_{\iota_Y} \\ 
}\]

\begin{lem}\label{lem:flatprojection}
The integral functor 
$\fm{\mc U_0}\colon D^b(X_0)\to D^b(Y_0)$
is a Fourier--Mukai transform, where $\mc U_0:=\LL\iota_Z^* \mc U_Z$. 
\end{lem}   

\begin{proof}
By the flat base change theorem
and the projection formula, we obtain
\begin{equation}\label{eqn:P_0}
\fm{\mc U_Z}(\iota_{X*} \alpha)\cong \iota_{Y*}\fm{\mc U_0}(\alpha)
\end{equation}
for all $\alpha \in D^b(X_0)$.
For a quasi-inverse functor $\fm{\mc U'}$ of $\fm{\mc U}$,
a similar statement is true. Hence
we know that 
$$\fm{\mc U'_0}\circ \fm{\mc U_0}(\mc O_s)\cong \mc O_s, \quad
\fm{\mc U_0}\circ \fm{\mc U'_0}(\mc O_t)\cong \mc O_t, 
$$ 
for all closed points $s\in X_0$ and $t\in Y_0$.
By \cite[3.3]{BM98} these conditions imply that 
$$
\fm{\mc U'_0}\circ \fm{\mc U_0}\in \Pic X_0, \quad
\fm{\mc U_0}\circ \fm{\mc U'_0}\in \Pic Y_0
$$
as autoequivalences of $D^b(X_0)$ and $D^b(Y_0)$.
Combining these with (\ref{eqn:P_0}), we conclude that 
$$
\fm{\mc U'_0}\circ \fm{\mc U_0}\cong \id,\quad 
\fm{\mc U_0}\circ \fm{\mc U'_0}\cong \id.
$$
\end{proof}

Let $S_i \to C$ $(i=1,\ldots,4)$ be flat projective 
morphisms between smooth projective varieties. 
We denote various closed embeddings by 
\begin{align*}
\iota _{ij}& \colon S_i\times_C S_j \hookrightarrow S_i\times S_j
\quad (i,j\in \{1,\ldots,4 \}), \\
\iota &\colon S_1\times_C S_3\times S_2\times_C S_4
 \hookrightarrow S_1\times S_2\times S_3\times S_4,\\
\iota _0   & \colon S_1\times_C S_2\times_C S_3\times_C S_4\hookrightarrow 
S_1\times S_2\times S_3\times S_4.
\end{align*}
 
%%%
%%%
%%% lemma: product

\begin{lem}\label{lem:product} 
Suppose that perfect objects 
$$
\mc P \in D_{\perf}(S_1\times_C S_3) 
\mbox{ and } \mc Q \in D_{\perf}(S_2\times_C S_4)
$$ 
give Fourier--Mukai transforms 
$$
\Phi ^{\mc P}\colon  D^b(S_1)\to D^b(S_3)
\mbox{ and }
\Phi ^{\mc Q}\colon  D^b(S_2)\to D^b(S_4)
$$
respectively. 
Assume that both of $S_1\times_C S_2$ and $S_3\times_C S_4$ are smooth.
Then the object 
$$
\mc P\boxtimes \mc Q\in D_{\perf}(S_1\times_CS_2\times S_3\times_CS_4)
$$ 
gives a Fourier--Mukai transform 
$$
\Phi ^{\mc P\boxtimes \mc Q}
\colon  D^b(S_1\times_C S_2)\to D^b(S_3\times_C S_4).
$$
\end{lem}

\begin{proof}
Under the assumptions in Lemma \ref{lem:product}, 
$\Phi ^{\iota_{13*}\mc P\boxtimes \iota_{24*}\mc Q}$ gives a 
Fourier--Mukai transform from
$D^b(S_1\times S_2)$ to $D^b(S_3\times S_4)$ 
(see \cite[Exercise 5.20]{Hu06}).
By the projection formula and the flat base change theorem, 
we have
\begin{align*}
{\iota_{13*}\mc P\boxtimes \iota_{24*}\mc Q}
=&\iota_{L*}\tilde{\pi}_{13}^*\mc P\ltensor 
\pi_{24}^*\iota_{24*}\mc Q \\
=&\iota_{L*}(\tilde{\pi}_{13}^*\mc P\ltensor 
\LL\iota_{L}^*\pi_{24}^*\iota_{24*}\mc Q) \\
=&\iota_{L*}(\tilde{\pi}_{13}^*\mc P\ltensor 
\iota_{R*} \tilde{\pi}_{24}^*\mc Q) \\
=&\iota_{L*}\iota_{R*}
(\LL\iota_{R}^* \tilde{\pi}_{13}^*\mc P\ltensor 
 \tilde{\pi}_{24}^*\mc Q) \\
=&\iota _*(\mc P\boxtimes \mc Q),
\end{align*}
where we consider the following diagram:
\[
\xymatrix{   S_1\times S_3 
& S_1\times S_2\times S_3\times S_4 \ar[l]_{\pi_{13}\qquad } \ar[dr]^{\pi_{24}}  
&     \\
  S_1\times_C S_3 \ar@{^{(}->}[u]^{\iota_{13}}
& S_1\times_C S_3 \times S_2\times S_4 \ar[l]_{\tilde{\pi}_{13}\qquad} 
\ \ar[r]_{\qquad \pi_{24}\circ \iota_L} 
\ar@{^{(}->}[u]^{\iota_{L}} 
& S_2\times S_4  \\
& S_1\times_C S_3 \times S_2\times_C S_4 
\ar[ul]^{\tilde{\pi}_{13}\circ \iota_R} 
\ \ar[r]^{\qquad \tilde{\pi}_{24}} 
\ar@{^{(}->}[u]^{\iota_{R}} 
& S_2\times_C S_4  \ar@{^{(}->}[u]^{\iota_{24}}  \\
}
\]

Apply Lemma \ref{lem:flatprojection} for 
$$
X=S_1\times S_2,\ Y=S_3\times S_4,\ 
Z=S_1\times_CS_3\times S_2\times_CS_4,\ \mc U_Z=\mc P\boxtimes \mc Q
$$
and 
$$
X_0=S_1\times _CS_2,\quad Y_0=S_3\times _CS_4,
\quad Z_0=S_1\times_CS_2\times_C S_3\times_CS_4
$$
to get the conclusion.
\end{proof}

%========================================================================
\subsection{Schoen's construction}\label{subsection:schoen}
Fix a positive integer $N$ and
choose an even integer $m\gg N$.
Take a rational  elliptic surface $\pi_0\colon B\to \PP^1$ with a section.
For a point $s\in\PP^1$,
let $\xi_s$ be an element of the order $m$ in $WC(B_s)$. 
We obtain from (\ref{eq:rational}) 
a rational  elliptic surface 
$\pi_1\colon S_1\to \PP^1$
admitting a unique multiple fiber over $s$ 
such that $J(S_1)\cong B$.
Theorem \ref{thm:FMelliptic0} implies that 
the set $\FM(S_1)$ contains at least $N$ elements $S_1,\ldots,S_N$.

%%%
%%%
%%%  lemma {lem:rank2}

\begin{lem}\label{lem:rank2}
For each $i,j$,
there is a vector bundle $\mc P_0$ on $S_i\times _{\PP ^1}S_j$ with 
$\rk \mc P_0=2$
such that the integral functor  $\fm{\mc P_0}\colon D^b(S_i)\to D^b(S_j)$ becomes a
Fourier--Mukai transform.
\end{lem}
\begin{proof}
By the choice of $S_i$ and $S_j$, we have 
$S_i\cong J^b(S_j)$ for some $b\in \Z$ such that $(b,m)=1$ by Theorem \ref{BMelliptic}.
In particular, \cite[Lemma 4.2]{BM01} implies that 
$J_{S_j}(2,b)\cong J^b(S_j)$, where $J_{S_j}(2,b)$ is defined in \S \ref{sub:FMellipitic}.
Take the universal sheaf $\mc P_0$ on 
$J_{S_j}(2,b) \times _{\PP ^1}S_j (\cong S_i\times _{\PP ^1}S_j)$. Then 
$\mc P_0$ satisfies the desired properties. See \cite[\S 4]{BM01} 
and \cite{Br98} for the details.

\end{proof}

Take another rational  elliptic surface 
$\pi\colon S\to \PP^1$.
% such that 
%$\pi$ has only irreducible fibers and no multiple fibers.
Let us denote the elliptic fibrations 
by $\pi_i\colon S_i\to \PP^1$ and define $X_i$ to be 
the fiber product of $S$ and $S_i$ over $\PP^1$ (cf. \cite{Sc88}). 
We define 
$p_i,p,f_i$ as follows.

\noindent
\[ \xymatrix{ & X_i \ar[dl]_{p_i} \ar[dr] ^{p} \ar[dd]_{f_i}&  \\
  S_i \ar[dr]_{\pi_i}&        & S\ar[dl]^{\pi} \\
     & \PP ^1 &  \\ 
}\]

\noindent
We also assume that 
\begin{itemize}%\label{enu:assumptions}
\item
$\Delta(\pi_1)\cap \Delta(\pi)$ is empty (equivalently, 
$\Delta(\pi_i)\cap \Delta(\pi)$ is empty for all $i$, 
since $\Delta(\pi_1)=\Delta(\pi_i)$), and
\item
the generic fibers of $\pi_i$ and $\pi$ are not isogenous each other. 
\end{itemize}
The first condition implies that $X_i$'s are smooth
and the second condition will be used when 
we apply the argument in \cite{Na91}.
The first condition is fulfilled by replacing $\pi$ with
the composition of $\pi$ and some automorphism on $\PP^1$.
The second condition is satisfied by choosing general $S$ 
in the family of elliptic surfaces.

%========================================================================
\subsection{Proof of Theorem \ref{thm:main20}}\label{subsection:main2}
In this subsection
we inherit all notations in \S \ref{subsection:schoen}.
Smooth projective varieties $X$ and $Y$ are said to be 
\emph{deformation equivalent} 
if there is a smooth proper holomorphic map between connected 
complex analytic spaces
$h \colon \mc X\to T$ such that each irreducible 
component of $T$ is smooth and
$\mc X_s\cong X$ and  $\mc X_t\cong Y$ for points $s,t\in T$.  

%%%
%%%
%%% lemma:   deformation equivalent 

\begin{lem}\label{lem:invariant}
All $X_i$'s are minimal with 
$\kappa (X_i)=1$ and 
they have the following Hodge diamond:
$$
\begin{array}{ccccccccc}
 &  &  & 1&  &  &  &  & \\
 &  & 0&  & 0&  &  &  & \\
 & 0&  &19&  & 0&  &  & \\
1&  &19&  &19&  & 1&  & \\
 & 0&  &19&  & 0&  &  & \\
 &  & 0&  & 0&  &  &  & \\
 &  &  & 1&  &  &  &  & \\ 
\end{array}
$$
Furthermore they are deformation equivalent and derived equivalent to
each other.
\end{lem}

\begin{proof}
Consider the morphism
$$
g_i (:=\pi_i \times \pi )\colon S_i\times S \to \PP^1\times \PP ^1
$$
and then we have $g_i^*\Delta_{\PP ^1} = X_i$.
Therefore the adjunction formula says that 
$K_{X_i}=K_{{S_i}\times S}+X_i|_{X_i}=rF_i$ for a general fiber 
$F_i$ of $f_i$ and some $r\in \Q_{>0}$. This implies that $\kappa (X_i)=1$.
By the use of the flat base change theorem, 
we can show 
\begin{align*}
\mb Rf_{i*}\mc O_{X_i}
\cong& \mb R\pi_{i*}\mc O_{S_i}\ltensor \mb R{\pi}_{*}\mc O_{S}\\
\cong& (\mc O_{\PP^1}\oplus \mc O_{\PP^1}(-1)[-1]) 
     \otimes (\mc O_{\PP^1}\oplus \mc O_{\PP^1}(-1)[-1]) 
\end{align*}
and hence obtain 
$$
\mb R^1f_{i*}\mc O_{X_i}\cong \mc O_{\PP^1}(-1)\oplus \mc O_{\PP^1}(-1) 
\text{ and }
\mb R^2f_{i*}\mc O_{X_i}\cong \mc O_{\PP^1}(-2).
$$
In particular, we  have $h^1(X_i,\mc O_{X_i})=h^2(X_i,\mc O_{X_i})=0$
and 
$h^3(X_i,\mc O_{X_i})=1$
by the Leray spectral sequence.
The Euler number $e(X_i)$ should be $0$, since the Euler number of 
every fiber of $f_i$ is $0$. 

The Picard number $\rho (X_i)(=h^{1,1}=h^{2,2})$ is $19$; 
more precisely there is a short exact sequence
\[
\xymatrix{ 
0\ar[r]& \Pic \PP ^1 \ar[r]^{\Pi\qquad }& \Pic S_i\times \Pic S
\ar[r]^{\qquad p_i^*\otimes p^*} &\Pic X_i\ar[r] &0,
}
\]
where $\Pi (M)=(\pi^*_iM , \pi^*M^{-1})$ for $M\in \Pic \PP ^1$.
The surjectivity of $p_i^*\otimes p^*$ is proved in the proof of 
\cite[Proposition 1.1]{Na91}.
The inclusion 
\begin{equation}\label{eqn:Pi}
\ker (p_i^*\otimes p^*)\subset \im \Pi
\end{equation}
is proved as follows:
First let us state the following claim due to Namikawa.
\begin{cla}[Proof of Proposition 1.1 in \cite{Na91}]\label{cla:namikawa}
Take $L_i\in \Pic S_i$ and $L\in\Pic S$ and suppose that
the line bundle $p_i^*L_i\otimes {p}^*L$ is an effective divisor on $X_i$
(henceforth we identify the isomorphism classes of 
line bundles with the linear equivalence classes of Cartier divisors).
Then there is an integer $a\in \Z$ such that 
both of $L_i\otimes\mc O_{\PP^1}(a)$ and 
$L\otimes \mc O_{\PP^1}(-a)$ are  
effective divisors.
\end{cla}
\noindent
Take $(L_i,L)\in \ker (p_i^*\otimes p^*)$.
We conclude from Claim \ref{cla:namikawa} that 
the pair $(L_i, L)$ is equal to a pair of some effective divisors
 in the group $(\Pic S_i\times \Pic S)/\Pic \PP ^1$.
In particular, to get (\ref{eqn:Pi}),
we may assume that $L_i$ and $L$ are effective divisors.
Then we can easily deduce $(L_i,L) \in \im \Pi$.
 The other inclusion of (\ref{eqn:Pi}) is obvious.

We also know $h^{1,2}(X_i)=h^{2,1}(X_i)=19$ 
by the equality $e(X_i)=2(h^{1,1}-h^{1,2})$.

Elliptic surfaces $S_i\to \PP^1$ ($i=1,\ldots, N$) are 
deformation equivalent through elliptic surfaces
(\cite[Theorem 1.7.6]{FM94}). In particular, all $X_i$'s are also 
deformation equivalent to each other.
 
 The fact that $X_i$'s are mutually derived equivalent follows from  
 Lemma \ref{lem:product}.
 Hence the last assertion follows.
\end{proof}

Before going further, we give a remark which is rather obvious from the above proof:  
In \cite{Sc88}, Schoen takes the fiber products of 
two rational  elliptic surfaces 
without multiple fibers and then
he obtains Calabi--Yau $3$-folds as the result.
In our construction,  
at least one of two elliptic surfaces has a 
multiple fiber. Consequently,
our $3$-folds have the Kodaira dimensions $1$ as above.

%%%
%%%
%%% lemma: non-birational

\begin{lem}\label{lem:birational}
$X_i$ and $X_j$ are not birationally equivalent for $i\ne j$.
\end{lem}

\begin{proof}
First we show that $X_i$ and $X_j$ are not isomorphic as follows.
Suppose that there is an isomorphism $\varphi\colon X_i\to X_j$.
Note that $f_i$ and $f_j$ are the Iitaka fibrations, that is,
they are defined by the complete linear system of some multiple 
of canonical divisors $K_{X_i}$ and $K_{X_j}$. In particular, 
there is an automorphism $\delta$ on $\PP^1$ such that 
$\delta\circ f_i=f_j\circ \varphi$.
Moreover we can see that the relative Picard numbers 
$\rho(X_i/\PP ^1)$ and $\rho (X_j/\PP ^1)$ are $2$. 
Thus $f_i$ (resp. $f_j$) factors through only in two ways;
$f_i$ (resp. $f_j$) factors through $S_i$ (resp. $S_j$) or $S$.
This is absurd by $S_i\not\cong S_j$.  

Next we show that $X_i$ has no small contractions for all $i$.
Suppose that $X_i$ has a small contraction contracting a 
curve $C$. Since $K_{X_j}\cdot C=0$, $C$ is also contracted by 
the Iitaka fibration $f_i$. But this contradicts $\rho (X_i/\PP ^1)=2$.

If minimal 3-folds $X_i$ and $X_j$ are birational, 
they are connected by a sequence of flops.
But it is impossible by the facts proved above.

\end{proof}

%%%
%%%
%%% Hodge isometry

\begin{lem}\label{lem:hodge}
There are Hodge isometries
$$
(H^3(X_i,\Z)_{\text{free}},Q_{X_i})\cong (H^3(X_j,\Z)_{\text{free}},Q_{X_j})
$$
for all $i,j$,
where the polarizations are given by the intersection forms.
\end{lem}

\begin{proof} 
Let $X$ and $Y$ be derived equivalent Calabi--Yau $3$-folds.
C\u{a}ld\u{a}raru shows the existence of Hodge isometries 
 
\begin{itemize}
\item between the free parts of $H^3(X,\Z[\frac{1}{2}])$
and $H^3(Y,\Z[\frac{1}{2}])$ (\cite[Proposition 3.1]{Ca07}), 
and 
\item between the free parts of $H^3(X,\Z)$
and $H^3(Y,\Z)$ under some additional assumptions 
(\cite[Proposition 3.4.]{Ca07}).
\end{itemize}

To get the conclusion as desired, 
we use and modify his argument in \cite{Ca07}.
Unlike $3$-folds $X,Y$ there, 
our $3$-folds $X_i$'s are not Calabi-Yau's,
and so $c_1(X_i)$ survives.
We have to take care of it.

First we put $X=X_i$ and $Y=X_j$ to adapt our notations with 
the one in \cite{Ca07}.

\emph{Step 1.}
As is well-known, a Fourier--Mukai transform 
$\fm{\mc U}\colon D^b(X) \to D^b(Y)$ with the kernel $\mc U\in D^b(X\times Y)$
induces an isometry
$$
\varphi:=\pi_{Y*}(\pi_X^*(-)\ch(\mc U)\sqrt{\td (X\times Y)})
 \colon H^*(X,\C)\to H^*(Y,\C).
$$
More precisely, $\varphi$ preserves the odd cohomologies 
and the Hochshild graded pieces
$\bigoplus _{q-p=k}H^{p,q}(X)$ and $\bigoplus _{q-p=k}H^{p,q}(Y)$ for all $k\in \Z$, 
which yields, as in the proof of 
\cite[Proposition 3.1]{Ca07}, that $\varphi$ restricts to 
a Hodge isometry 
$$
\varphi|_{H^3(X,\C)}\colon H^3(X,\C)\to H^3(Y,\C),
$$
preserving the intersection forms.
Indeed, in order to show this fact in \cite[ibid.]{Ca07}, 
$X$ and $Y$ are not needed 
to be Calabi-Yau's but they satisfy the equations $h^1(\mc O)=h^2(\mc O)=0$, 
which are true in our situation.

\emph{Step 2.}
Below we denote by $\alpha^{s,t}$ 
the $(s,t)$-K\"unneth components in $H^s(X,\Q)\otimes H^t(Y,\Q)$ of 
$\alpha\in H^{s+t}(X\times Y,\Q)$. 
By the argument in \cite{Ca07},
we know that 
\begin{equation*}
\varphi|_{H^3(X,\C)}(-)
=\pi_{Y*}(\pi_X^*(-)(\ch (\mc U)\sqrt{\td (X\times Y)})^{3,3}).
\end{equation*}
So we want to get more information of  
$
(\ch (\mc U)\sqrt{\td (X\times Y)})^{3,3}
$ 
for our purpose.

We have 
\begin{align*}
\sqrt{\td (X\times Y)}=
&1+\frac{1}{4}c_1(X\times Y)+(\frac{1}{96}c_1(X\times Y)^2+\frac{1}{24}c_2(X\times Y))\\
&+\frac{1}{96}c_1(X\times Y)c_2(X\times Y)+\text{ higher order terms}.
\end{align*}
Hence there are no $(3,3),(3,0),(0,3)$-components
in $\sqrt{\td (X\times Y)}$. 
In addition, because $H^1(X,\Q)=H^1(Y,\Q)=0$
we have 
$$
(\ch (\mc U)\sqrt{\td (X\times Y)})^{3,3}= 
(\ch (\mc U))^{3,3}(\sqrt{\td (X\times Y)})^{0,0}
$$
and 
\begin{equation}\label{eqn:33}
(\ch(\mc U))^{3,3}=\frac{1}{6}(c_1(\mc U)^3-3c_1(\mc U)c_2(\mc U)+3c_3(\mc U))^{3,3}
=\frac{1}{2}{c_3(\mc U)}^{3,3}.
\end{equation}
In particular, we obtain
\begin{equation}\label{eqn:H^3}
\varphi|_{H^3(X,\C)}(-)
=\frac{1}{2}\pi_{Y*}(\pi_X^*(-){c_3(\mc U)}^{3,3}).
\end{equation}

\emph{Step 3.}
In this step, we show that 
$\frac{1}{2}{c_3(\mc U)}^{3,3}\in H^6(X\times Y, \Z)$ if we choose an object
$\mc U$ appropriately. 
With (\ref{eqn:H^3}), this yields the conclusion.
This time, we use ideas in the proof of \cite[Proposition 3.4]{Ca07}.  

Put $Z=X\times_{\PP^1} Y$ and denote 
the closed embedding by
$\iota\colon Z\hookrightarrow X\times Y$.
We choose a vector bundle $\mc P_0$ on 
$S_i\times_{\PP ^1} S_j$ with  $\rk\mc P_0=2$ as in  Lemma \ref{lem:rank2}. 
Define a sheaf $\mc U_0$ on $Z$ as  
$\mc U_0:=\mc P_0\boxtimes \mc O_{\Delta_{S}}.$
Hence 
$$
c_1(\mc U_0)=\rk(\mc P_0)\tilde{\pi}^*c_1(\mc O_{\Delta_{S}})
=2\tilde{\pi}^*c_1(\mc O_{\Delta_{S}}),
$$ 
where
$\tilde{\pi}\colon Z\to S\times_{\PP ^1}S$ is the projection.
Then as in the proof of  Lemma \ref{lem:product},
the object $\mc U:=\iota_* \mc U_0$ gives the kernel 
of the Fourier--Mukai transform $\fm{\mc U}$. 
By the Grothendieck--Riemann--Roch theorem, we have 
\begin{align}\label{al:Grot}
\ch (\mc U)=\iota_* (\ch (\mc U_0)\td (\mc N_{Z/X\times Y})^{-1})
\end{align}  
(cf. \cite[pp. 283]{Fu98}). 

Define $g:=f_i\times f_j \colon X\times Y \to \PP ^1\times \PP ^1$.
Then we have $g^*\Delta_{\PP ^1} =Z$
and hence  
$$
\mc N_{Z/X\times Y}=g^*\mc N_{\Delta /\PP ^1\times \PP ^1}
=g^*\mc O_{\PP  ^1}(2).
$$
In particular, 
$
\td (\mc N_{Z/X\times Y})^{-1}=1-g^*c_1(\mc O_{\PP^1}(1)).
$
%where $H$ is the class  in $H^2(\PP^1,\Z)$ corresponding to a point in $\PP^1$.

Taking the $(3,3)$-components of both sides of (\ref{al:Grot}), 
we obtain from (\ref{eqn:33})  
\begin{align*}
&\frac{1}{2}{c_3(\mc U)}^{3,3}
=(\iota_* (\ch (\mc U_0)\td (\mc N_{Z/X\times Y})^{-1}))^{3,3}\\
=&(\iota_* ((r(\mc U_0)+c_1(\mc U_0)+\frac{1}{2}c_1(\mc U_0)^2-c_2(\mc U_0)
+\text{h.o.t.})
(1-g^*c_1(\mc O_{\PP^1}(1)))))^{3,3} \\
=&(\iota_* (\frac{1}{2}c_1(\mc U_0)^2-c_2(\mc U_0)-c_1(\mc U_0)g^*c_1(\mc O_{\PP^1}(1))))^{3,3}.
\end{align*}
Therefore we obtain 
$\frac{1}{2}{c_3(\mc U)}^{3,3}\in H^6(X\times Y, \Z)$. 
\end{proof}

\noindent
\emph{Proof of Theorem \ref{thm:main20}.}
Combining Lemmas  \ref{lem:invariant}, \ref{lem:birational} and
\ref{lem:hodge}, we obtain Theorem \ref{thm:main20}.
\qed

\begin{remark}\label{rem:Chakiris}
In \cite{Ch80}, Chakiris gives  a sketchy proof of the following
result:
Let $S$ be a simply connected, minimal elliptic surface with 
$p_g(S)\ne 0$, having one or at most two multiple fibers.  
Then the period map has a positive dimensional fiber at the point 
corresponding to the surface $S$. 

Because it is known that the Fourier--Mukai number $|\FM (S)|$ is finite (\cite{BM01}),
general elements in the fiber do not have equivalent derived categories.
I am not sure whether we can show Lemma \ref{lem:hodge} 
without using derived equivalence. 
\end{remark}

\begin{remark}
Let $X$ and $Y$ be birationally equivalent smooth minimal 
 $3$-folds. Bridgeland theorem \cite{Br02} says that 
 $X$ and $Y$ are derived equivalent. 
 Assume furthermore  that 
 $h^1(\mc O_X)=h^1(\mc O_Y)= 0$.
Then Koll\'ar \cite{Ko89} proves that
there is a rational polarized Hodge isometry
$$
(H^3(X,\Q),Q_{X})\cong (H^3(Y,\Q),Q_{Y}).
$$  
Hence 
the non-birationality (Lemma \ref{lem:birational})
 makes Theorem \ref{thm:main20} novel.
\end{remark}

%%%%%%%%%%%%%%%%%%%%%%%%%%%%%%%%%%%%%%%%%%%%%%%%%%%%%%%%%%%%%%%%%%%%%%%%%%%%%%%

\indent

\medskip
\noindent 
\textsc{Department of Mathematics
and Information Sciences,\\
Tokyo Metropolitan University,\\
1-1 Minamiohsawa,
Hachioji-shi,
Tokyo,
192-0397,
Japan \\
E-mail:}
\texttt{hokuto@tmu.ac.jp}


\begin{thebibliography}{Fri95}

\bibitem[BC09]{BC09}
L. Borisov, A. C\u{a}ld\u{a}raru, 
The Pfaffian-Grassmannian derived equivalence. J. Algebraic Geom. 18 (2009),
201--222. 
\bibitem[Br98]{Br98}
T. Bridgeland, Fourier--Mukai transforms for elliptic surfaces. J. Reine Angew. Math. 498 (1998), 115--133. 

\bibitem[Br02]{Br02}
T. Bridgeland, Flops and derived categories. Invent. Math. 147 (2002), 613--632. 

\bibitem[BM98]{BM98}
T. Bridgeland, A. Maciocia,
Fourier-Mukai transforms for quotient varieties. math.AG/9811101.

\bibitem[BM01]{BM01}
T. Bridgeland, A. Maciocia,
Complex surfaces with equivalent
 derived categories. Math. Z. 236 (2001), 677--697. 


\bibitem[Ca07]{Ca07}
A. C\u{a}ld\u{a}raru,
Non-birational Calabi-Yau threefolds that are derived equivalent.
Internat. J. Math. 18 (2007), 491--504. 


\bibitem[Ch80]{Ch80}
K. Chakiris,
Counterexamples to global Torelli for certain simply connected surfaces.
Bull. Amer. Math Soc. (N.S.) 2 (1980), 297--299.

\bibitem[CD89]{CD89}
F. Cossec, I. Dolgachev,
 Enriques surfaces. I. Progress in Mathematics, 76. Birkhauser Boston, Inc., 
Boston, MA, 1989. x+397 pp.

\bibitem[Do81]{Do81}
I. Dolgachev,
Algebraic surfaces with $p_g=q=0$. in "Algebraic surfaces", 
Proc. CIME Summer School in Cortona, Liguore, Napoli, 1981, 97--216.


%\bibitem[Fri98]{Fri98}
%R. Friedman, Algebraic surfaces and holomorphic vector bundles. Universitext. 
%Springer-Verlag, New York, 1998.


\bibitem[FM94]{FM94}
R. Friedman, J. Morgan, Smooth four-manifolds and complex surfaces. 
Ergebnisse der Mathematik und ihrer Grenzgebiete (3),
27. Springer-Verlag, Berlin, 1994. x+520 pp.
 
 
\bibitem[Fu98]{Fu98}
W. Fulton, Intersection theory. Second edition. 
Ergebnisse der Mathematik und ihrer Grenzgebiete. 3. Folge. 
%A Series of Modern Surveys in Mathematics 
%[Results in Mathematics and Related Areas. 3rd Series.
% A Series of Modern Surveys in Mathematics], 2.
 Springer-Verlag, Berlin, 1998. xiv+470 pp.

\bibitem[Hu06]{Hu06}
D. Huybrechts, Fourier-Mukai transforms in algebraic geometry.
 Oxford Mathematical Monographs. The Clarendon Press, 
 Oxford University Press, Oxford, 2006. viii+307 pp.
 
\bibitem[HLOY03]{HLOY03}
S. Hosono, B. Lian, K. Oguiso, S.-T. Yau, 
Kummer structures on $K3$ surface: an old question of T. Shioda. Duke Math. J. 120 (2003), 635--647. 

\bibitem[Ka02]{Ka02}
Y. Kawamata, D-equivalence and K-equivalence. J. Differential Geom. 61 (2002), 147-171.

\bibitem[Ko89]{Ko89}
J. Koll\'ar,
Flops. Nagoya Math. J. 113 (1989), 15--36. 

\bibitem[Na91]{Na91}
Y. Namikawa, 
On the birational structure of certain Calabi-Yau threefolds.
 J. Math. Kyoto Univ. 31 (1991), 151--164.


\bibitem[Na02]{Na02}
Y. Namikawa,
Counter-example to global Torelli problem for irreducible symplectic 
manifolds.  Math. Ann.  324  (2002), 841--845; 
Erratum. ibid., 847.

\bibitem[Og02]{Og02}
K. Oguiso, 
K3 surfaces via almost-primes. Math. Res. Lett. 9 (2002), 47--63.

\bibitem[Or96]{Or96}
D. Orlov, Equivalences of derived categories and $K3$ surfaces. 
Algebraic geometry, 7. J. Math. Sci. (New York) 84 (1997), 1361--1381.

\bibitem[Or02]{Or02}
D. Orlov, Derived categories of coherent sheaves on abelian varieties 
and equivalences between them. (Russian) Izv. Ross. Akad. Nauk Ser. Mat. 66 (2002), 131--158; translation in Izv. Math. 66 (2002), 569--594. 
 
\bibitem[Pe90]{Pe90}
U. Persson, Configurations of Kodaira fibers on rational elliptic surfaces. Math. Z. 205 (1990), 1--47.

\bibitem[Sc88]{Sc88}
C. Schoen, 
On fiber products of rational elliptic surfaces with section.
Math. Z. 197 (1988), 177--199. 

\bibitem[Se02]{Se02}
J.-P. Serre, Galois cohomology. Springer Monographs in Mathematics.
 Springer-Verlag, Berlin, 2002. x+210 pp. 

\bibitem[Sz04]{Sz04}
B. Szendr\H{o}i,
On an example of Aspinwall and Morrison.  Proc. Amer. Math. Soc.  132
(2004), 621--632.

\bibitem[To06]{To06}
Y. Toda, 
Fourier-Mukai transforms and canonical divisors. Compos. Math. 142 (2006),
962--982.
 
\bibitem[Ue04]{Ue04}
H. Uehara, An example of Fourier-Mukai partners of minimal elliptic surfaces. Math. Res. Lett. 11 (2004), 371--375.



\end{thebibliography}
\end{document}